\documentclass{article}

\usepackage{amssymb,amsmath,color,graphicx,dsfont,mathrsfs,amsthm,stmaryrd}
\usepackage{mathtools}
\usepackage{bbold}
\usepackage{authblk}
\usepackage{hyperref}
\usepackage[capitalise]{cleveref}
\hypersetup{linktocpage}

\usepackage[a4paper]{geometry}
\usepackage[
    backend=biber,
    maxbibnames=6,
    maxcitenames=6,
    style=alphabetic,
    sorting=nyt,
    isbn=false,
    doi=false,
    url=false,
    firstinits=true,
    date=year]{biblatex}
\AtEveryBibitem{\clearfield{pages}}

\usepackage[colorinlistoftodos,french]{todonotes}
\newtheorem{theorem}{Theorem}

\newtheorem{lemma}[theorem]{Lemma}
\newtheorem{proposition}[theorem]{Proposition}

\newtheorem{remark}{Remark}

\DeclareMathOperator\erf{erf}

\newcommand{\R}{\mathbb{R}}
\newcommand{\N}{\mathbb{N}}
\newcommand{\eps}{\epsilon}

\newcommand{\un}[1]{\underline{#1}}

\title{Small-time global null controllability of generalized Burgers' equations}
\author{Rémi Robin\footnote{remi.robin@inria.fr}}
\affil{Laboratoire Jacques-Louis Lions, Sorbonne Université, Paris, France}
\begin{document}
\maketitle
\begin{abstract}
    In this paper, we study the small-time global null controllability of the generalized Burgers' equations $y_t + \gamma |y|^{\gamma-1}y_x-y_{xx}=u(t)$ on the segment $[0,1]$. The scalar control $u(t)$ is uniform in space and plays a role similar to the pressure in higher dimension. We set a right Dirichlet boundary condition $y(t,1)=0$, and allow a left boundary control $y(t,0)=v(t)$. Under the assumption $\gamma>3/2$ we prove that the system is small-time global null controllable. Our proof relies on the return method and a careful analysis of the shape and dissipation of a boundary layer.
\end{abstract}
\tableofcontents

\section{Introduction}
\subsection{Description of the system}
For a given $T>0$, we are concerned with the following generalized Burgers' equations on the segment $[0,1]$:
\begin{equation}
    \label{eq:gen_burgers}
    \tag{$E_\gamma$}
    \left\{
        \begin{array}{ll}
            y_t+\gamma |y|^{\gamma-1} y_x-y_{xx}=u(t) \quad & \text{on }(0,T)\times (0,1), \\
            y(t,0)=v(t)& \text{on }(0,T),\\
            y(t,1)=0& \text{on }(0,T),\\
            y(0,x)=y_0(x)&\text{on }(0,1) ,
        \end{array}
    \right.
\end{equation}
and 
\begin{equation}
    \label{eq:gen_burgers_sign}
    \tag{$F_\gamma$}
    \left\{
        \begin{array}{ll}
            y_t+\gamma \operatorname{sign}(y)|y|^{\gamma-1} y_x-y_{xx}=u(t) \quad & \text{on }(0,T)\times (0,1), \\
            y(t,0)=v(t)& \text{on }(0,T),\\
            y(t,1)=0& \text{on }(0,T),\\
            y(0,x)=y_0(x)&\text{on }(0,1) ,
        \end{array}
    \right.
\end{equation}
where $u(t)$ is an interior control which does not depend on space, and $v(t)$ is a boundary control. We are interested in the small-time global null controllability. That is, for any initial (possibly large) datum $y_0$  and any (possibly small) final time $T$, can we find some controls $u$ and $v$ such that the solution of \cref{eq:gen_burgers} or \cref{eq:gen_burgers_sign} is steered to $0$ in time $T$?

\subsection{Motivation and existing results}

The main motivation for studying systems \eqref{eq:gen_burgers} and \eqref{eq:gen_burgers_sign} comes from controllability questions in fluid mechanics, in both dimensions 2 and 3. Let us briefly recall some major results obtained during the last thirty years. Coron and Glass have respectively established the two-dimensional \cite{coronControlabiliteExacteFrontiere1993} and three-dimensional \cite{glassControlabiliteExacteFrontiere1997} controllability of the Euler system. As for the Navier--Stokes equation, the two-dimensional case in a manifold without boundary has been tackled by Coron and Fursikov in \cite{coronGlobalExactControllability1996}. Following the later result, Fursikov and Imanuvilov proved a global exact controllability result in three dimensions in \cite{fursikovExactControllabilityNavierStokes1999} with a control acting on the whole boundary.

 More recently, Coron, Marbach and Sueur proved in \cite{coronSmalltimeGlobalExact2020} the small-time global null controllability relying on a control acting on part of the boundary and a Navier slip-with-friction condition elsewhere. This result was extended to the case of smooth solutions by Liao, Sueur and Zhang in \cite{liaoSmoothControllabilityNavier2022}.The same problem with Dirichlet boundary condition, stated by Lions in \cite{lionsExactControllabilityDistributed1991}, remains open and is considered to be a major challenge in the field.
 Most of the difficulties come from the interaction of the inertial and viscous forces near the uncontrolled boundary, which creates a difficult-to-control boundary layer. That question has motivated the study of several models with simpler geometry in two and three dimensions, for example \cite{chapoulyGlobalNullControllability2009}, \cite{guerreroRemarksGlobalApproximate2006}, \cite{guerreroResultConcerningGlobal2012} and \cite{liaoGlobalControllabilityNavier2022}.

In this work, we consider the generalized Burgers' equations \eqref{eq:gen_burgers} and \eqref{eq:gen_burgers_sign}. Those equations are a family of one-dimensional non-linear evolution partial differential equations, whose solutions exhibit a boundary layer behaviour near a Dirichlet boundary condition.
From an historical point of view, the classical viscous Burgers' equation ($F_2$) has been introduced and studied by Burgers in \cite{burgersMathematicalModelIllustrating1948}. Burgers' equation naturally appears among other fields in plasma physics, fluid dynamics and trafic flow. The generalized Burgers' equations are a generalisation of the classical Burgers' equation remaining in the class of viscous conservation law. A general study of these equations as well as the physical motivations behind them can be found for example in \cite{murrayPerturbationEffectsDecay1970,murrayGunnEffectOther1970,sachdevGeneralizedBurgersEquations1987,sunMetastabilityGeneralizedBurgers1999,escobedoAsymptoticBehaviourSourcetype1993}.

Let us now recall some major results concerning the controllability of Burgers' equation. Fursikov and Imanuvilov in \cite{fursikovControllabilityEvolutionEquations1996} proved the small-time local exact controllability in the vicinity of the trajectories. Their result relies on Carleman estimates, and only uses one boundary control. Its extension to generalized Burgers' equations is straightforward (see \cref{subsec:local_exact}).

Global controllability in finite time toward steady states with boundary controls was established in \cite{fursikovControllabilityCertainSystems1995}. Some generalisations are available: let us cite for example the work of Léautaud who found an extension for a wide class of viscous conservation laws in \cite{leautaudUniformControllabilityScalar2012} and the work of Ara\'ujo et al. on Burgers-$\alpha$ systems \cite{araujoUniformControllabilityFamily2021}. Results concerning the small-time global stabilization with an additional internal control were obtained recently in \cite{coronSmalltimeGlobalStabilization2021}.

An obstruction to the small-time global null controllability with boundary controls was found in \cite{guerreroRemarksGlobalControllability2007}, while an obstruction to small-time local null controllability with only a uniform in space internal control has been proved in \cite{fredericmarbachObstructionSmallTime2018}. Hence, the use of the internal control $u(t)$ together with a boundary control is necessary in \cref{eq:gen_burgers,eq:gen_burgers_sign}. With the help of this new control and two boundary controls, Chapouly proved the small-time global null controllability using results on the inviscid Burgers equation in \cite{chapoulyGlobalControllabilityNonviscous2009}.

Afterwards, Marbach extended the proof of the small-time global null controllability without the right boundary control, i.e. in our setting described by Eq. ($F_2$). A key feature of his proof is the Cole-Hopf transform (\cite{coleQuasilinearParabolicEquation1951,hopfPartialDifferentialEquation1950}). This transformation reduces Burgers' equation to the heat equation thanks to a change of variable.

In this paper, we generalize Marbach's work to the generalized Burgers' equations. To the best of our knowledge, those equations do not have an equivalent transformation. As a consequence, deriving precise estimates on the boundary layer is a more delicate task. We hope that this work will help to prepare the way for the many challenging problems remaining in control of fluid mechanics equations. In this scope, some related open problems are listed in \cref{sec:ccl_and_open}.

\subsection{Rigorous statement of our main result}
We have to provide a reasonable definition for the solutions of \cref{eq:gen_burgers}. Let us fix the final time $T\in \R^+$ and consider the following generalisation:
\begin{equation}
    \label{eq:gen_burgers generalized}
    \tag{$G_{\gamma}$}
    \left\{
    \begin{array}{ll}
        y_t+\gamma |y|^{\gamma-1} y_x-y_{xx}=u(t) \quad & \text{on }(0,T)\times (0,1), \\
        y(t,0)=v(t)& \text{on }(0,T), \\
        y(t,1)=w(t)& \text{on }(0,T),  \\
        y(0,x)=y_0(x)&\text{on }(0,1), 
    \end{array}
    \right.
\end{equation}
where the scalar controls are $u\in L^{\infty}(0,T)$ and $v,w\in H^{1/4}(0,T)\cap L^\infty(0,T)$. We also assume $\gamma \geq 1$ and $y_0 \in L^\infty(0,1)$, then there exists a unique solution of \eqref{eq:gen_burgers generalized} in $\mathscr{C}^0([0,T];L^2)\cap L^2(0,T; H^1)\cap L^\infty((0,T)\times (0,1))$. Note that $H^{1/4}$-regularity on time for the boundary controls is a natural assumption, as it corresponds to the minimum regularity that ensures well posedness of the usual 1D heat equation in $\mathscr{C}^0([0,T];L^2)\cap L^2(0,T; H^1)$ \cite[chapter 4]{MR0350178}. We do not provide the proof of the well-posedness here, it is based on \textit{a priori} energy estimates and a fixed point argument (see \cite{lionsQuelquesMethodesResolution1969}). For the interested reader, we mention that well-posedness in less regular space is studied in \cite{bekiranovInitialvalueProblemGeneralized1996,liWellposednessGeneralizedBurgers2019}. A similar existence statement can be stated for \cref{eq:gen_burgers_sign} with 3 controls.

Let us now state our main contribution.
\begin{theorem}
    \label{th:main}
    Suppose $\gamma>3/2$, $y_0\in L^\infty(0,1)$ and $T>0$. Then, there exist $u\in L^\infty(0,T)$ and $v\in H^{1/4}(0,T)\cap L^\infty(0,T)$ steering the solution $y$ of \eqref{eq:gen_burgers} to the null state in time $T$.

    If moreover $\gamma\geq 2$, then there exist $u\in L^\infty(0,T)$ and $v\in H^{1/4}(0,T)\cap L^\infty(0,T)$ steering the solution $y$ of \eqref{eq:gen_burgers_sign} to the null state in time $T$.
\end{theorem}
We provide only the proof for \cref{eq:gen_burgers}. Indeed, the adaptation of our proof for \cref{eq:gen_burgers_sign} is straightforward except for \cref{sec:passive_stage}. In that section, our proof does not extend to $3/2< \gamma \leq 2$ (there is a sign issue in \cref{eq:estimate_weighted}). The case ($F_2$) follows from \cite{marbachSmallTimeGlobal2014}.
\begin{remark}
    We are not able to tackle the case $\gamma \in (1,3/2]$ for \cref{eq:gen_burgers}, and we believe that our method cannot be used to tackle this (entire) range of $\gamma$. We comment this point in \cref{sec:ccl_and_open}.
\end{remark}

The sketch of the proof is the following:
to reach the null state in arbitrary small-time, we take advantage of the return method introduced by Coron in \cite{coronGlobalAsymptoticStabilization1992} (see also \cite{coronControlNonlinearity2009}),
and more specifically, the three-stages strategy developed by Marbach in \cite{marbachSmallTimeGlobal2014}.
\begin{itemize}
    \item Hyperbolic stage, first part: we introduce the steady state $\vartheta$ of \eqref{eq:gen_burgers} with $u=0$ and $v=\theta \gg 1$:
    \begin{align}
        \label{def:k_M}
        \left\{
        \begin{array}{l}
            \vartheta_{xx}=(\vartheta^\gamma)_x, \\
            \vartheta(0)=\theta, \quad
            \vartheta(1)=0 .
        \end{array}
        \right.
    \end{align}
    Note that $\vartheta$ exhibits a boundary layer near the right endpoint. Using the hyperbolic nature of the equation when it is governed by the non-linear term, we prove that we can stear the system to a neighborhood of $\vartheta$ in small-time.
    \item Hyperbolic stage, second part: we use the pressure-like term to drive our system to a neighborhood of the null state up to a boundary residue around $x=1$.
    \item Passive stage: we do not apply any control and wait for the dissipation of the boundary residue in small-time. The assumption $\gamma >3/2$ is crucial for this stage.
    \item Parabolic stage: we provide a local exact controllability result around zero using a fixed point method.
\end{itemize}
The proof of \cref{th:main} is then obtained by the combination of these stages resumed in \cref{prop:main_first_stage,lem:neighborhood_of_zero,prop:main_stage2,lem:local_exact}.
\subsection{Preliminaries}
\subsubsection{Comparison principle}
Let us recall the comparison principle for semi-linear parabolic equations. Let us suppose that $y$ (resp. $\tilde y$) is a solution of \eqref{eq:gen_burgers generalized} with controls $u,v,w$ (resp. $\tilde{u},\tilde{v},\tilde{w}$) and initial condition $y_0$ (resp. $\tilde y_0$), such that a.e.,
\begin{align*}
    \begin{array}{rlrl}
    u&\leq \tilde u,&\qquad 
    v&\leq \tilde v, \\
     w&\leq \tilde w,&
    y_0&\leq \tilde y_0.
    \end{array}
\end{align*}
Then,
\begin{align*}
    y\leq \tilde y.
\end{align*}
We refer to \cite{pucciMaximumPrinciple2007}. One can also easily extend the proof given by Marbach in \cite{marbachSmallTimeGlobal2014}.
\subsubsection{Study of the steady states}
It is crucial in our proof to have a reasonable description of the steady states $\vartheta$ defined by \cref{def:k_M}.
\begin{lemma}
    Let $\theta>0$, then \cref{def:k_M} admits a unique solution. Besides,
    \begin{align}
        \label{eq:k_M_with_C}
        &\vartheta_x=|\vartheta|^\gamma+\vartheta_x(1) \quad \text{ with } -\theta^\gamma-\theta < \vartheta_x(1) < -\theta^\gamma
    \end{align}
    and $\theta \mapsto \vartheta_x(1)$ is decreasing.
\end{lemma}
\begin{proof}
    To prove the existence, let us consider the application $\mathcal{F}$ given as follows: for any $(\theta,C)\in \R^2$, we associate the (local) solution $y$ of 
    \begin{align}
        \label{eq:tmp_steady_state}
        y(0)=\theta, &\quad  y_x=|y|^\gamma+C.
    \end{align}

    First, note that for any $\theta>0$, $\mathcal{F}(\theta,-\theta^\gamma)$ is a constant function. Besides, for any $\theta>0$, the function $\mathcal{F}(\theta,-\theta^\gamma-\theta)$ is strictly decreasing and bounded by below by $-(\theta^\gamma+\theta)^{1/\gamma}$. Let us define
    $$x^*= \inf \{x\geq 0 \mid \mathcal{F}(\theta,-\theta^\gamma-\theta)(x^*)\geq 0\}.$$
    One easily gets that for $x\in [0,x^*]$, $\mathcal{F}(\theta,-\theta^\gamma-\theta)(x)< \theta(1-x)$. Hence, $\mathcal{F}(\theta,-\theta^\gamma-\theta)(1)<0$.
    As a consequence $0$ belongs to the range of the continuous function $[-\theta^\gamma-\theta,-\theta^\gamma]\ni C \mapsto \mathcal{F}(\theta,C)(1)$. Thus, we can find $C^*\in [-\theta^\gamma-\theta,-\theta^\gamma]$ such that $\mathcal{F}(\theta,C^*)(1)=0$ and $x\mapsto \mathcal{F}(\theta,C^*)(x)$ is a solution of \cref{def:k_M} and \cref{eq:k_M_with_C} with $C^*=\mathcal{F}(\theta,C^*)_x(1)$. This concludes the proof of existence.

    Note also that  the function $C \mapsto \mathcal{F}(\theta,C)(1)$ is increasing. Thus, we proved the uniqueness of the solution of \cref{def:k_M}.

    Moreover, if one fix $C\in \R $ and consider $\theta_1<\theta_2$, $\mathcal{F}(\theta_1,C)<\mathcal{F}(\theta_2,C)$. Hence, $\theta \mapsto \vartheta_x(1)$ is decreasing.
\end{proof}
It is also easy to check that the steady state $\vartheta$ is a non-negative, decreasing and concave function.

\begin{remark}
    \label{rmk:boundary_layer_exp}
    \leavevmode
    \begin{itemize}
        \item In the case of the usual Burgers' equation,
        \begin{align}
            \vartheta(x)=\hat \theta \tanh(\hat \theta(1-x)),
        \end{align}
        where $\hat \theta$ is the unique solution of $\hat \theta \tanh(\hat \theta)=\theta$.
        \item For the linear case, i.e. with $\gamma=1$ (which is not included in the assumption of Theorem \ref{th:main}), we have
        \begin{align}
            \vartheta(x)=\frac{(e-e^x)}{e-1}\theta.
        \end{align}
    \end{itemize}
\end{remark}
For a general $\gamma$, we are not aware of an explicit expression for $\vartheta$. If $\gamma>1$, $\vartheta$ presents a boundary layer at the right endpoint (see \cref{fig:steady_states}) which is characterized in the next lemma.

\begin{figure}
    \includegraphics[width=\textwidth]{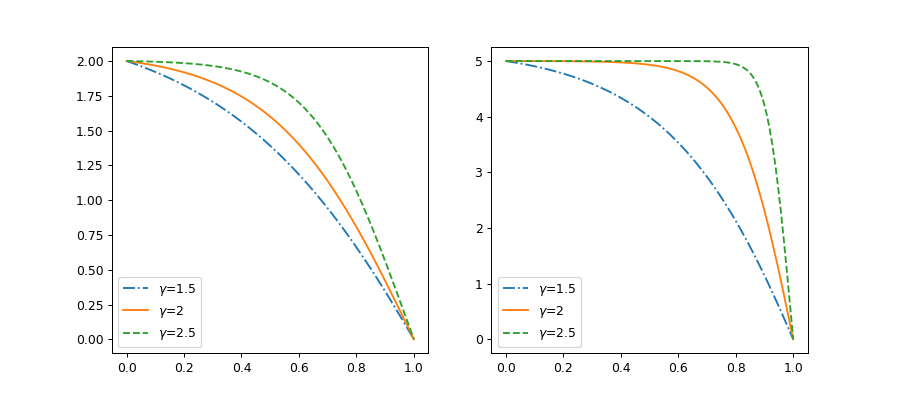}
    \caption{Steady states $\vartheta$ with $\theta=2$ and $\theta=5$ for different values of $\gamma$.}
    \label{fig:steady_states}
\end{figure}

\begin{lemma}
    \label{lem:boundary_layer}
    Let $\gamma>1$ and $\vartheta$ the solution of \cref{def:k_M}. Then $\vartheta$ exhibits a boundary layer of size $\frac{1}{\theta^{\gamma-1}}$ around $x=1$. More precisely, the following estimates hold:
    
    \begin{enumerate}
        \item Let $a \in (0,1)$, then for all $ \theta>0$,
        \begin{align}
            \vartheta(x)\geq a \theta, \quad  \forall x \in [0,1- \frac{a}{1-a^\gamma}\theta^{1-\gamma}].
        \end{align}
        \item Let $\eps>0$ and $\alpha< \gamma-1$, then there exists $C_{\eps,\alpha} >0$ such that
        \begin{align}
             \vartheta(x)\geq \theta-\eps , \quad  \forall x \in [0,1- C_{\eps,\alpha}\theta^{-\alpha}].
        \end{align}
    \end{enumerate}
\end{lemma}
\begin{proof}
    For $x\in (0,1]$, according to \cref{eq:k_M_with_C},
    \begin{align*}
    \vartheta_x(x)=\vartheta(x)^\gamma+\vartheta_x(1) \leq \vartheta(x)^\gamma- \theta^\gamma < 0.
    \end{align*}
    Thus, for any $x^*\in (0,1)$,
    \begin{align*}
        \int_{x^*}^1 \frac{-\vartheta_x(x) dx}{\theta^\gamma-\vartheta^\gamma(x)}\geq 1-x^*.
    \end{align*}
    By the change of variable formula with $z=\vartheta(x)$,
    \begin{align}
        \label{eq:proof_boundary_layer}
        \int_{0}^{\vartheta(x^*)} \frac{dz}{\theta^\gamma-z^\gamma}\geq 1-x^*.
    \end{align}
    To prove the first point, let $x^*\in (0,1)$ be the unique solution to the equation $\vartheta(x^*)=a \theta $. Injecting into Equation \eqref{eq:proof_boundary_layer}, we get
    \begin{align*}
        a\theta \frac{1}{\theta^\gamma(1-a^\gamma)} \geq \int_{0}^{\vartheta(x^*)} \frac{dz}{\theta^\gamma-z^\gamma} \geq 1-x^*.
    \end{align*}
    Thus,
    \begin{align*}
        x^* \geq 1- \frac{a}{1-a^\gamma}\theta^{1-\gamma}.
    \end{align*}

    For the second point, let $x^*\in (0,1)$ be the unique solution to the equation $\vartheta(x^*)=\theta- \eps$. Let us use \eqref{eq:proof_boundary_layer} again, and apply the change of variable $s=\frac{z}{\theta}$,

    \begin{align*}
    \int_{0}^{\vartheta(x^*)} \frac{dz}{\theta^\gamma-z^\gamma}=\theta^{1-\gamma}\int_0^{1-\eps/\theta} \frac{ds}{1-s^{\gamma}}.
    \end{align*}
    For $\gamma >1$, we have $\frac{1}{1-s} > \frac{1}{1-s^{\gamma}}$.
    Thus,
    \begin{align*}
        -\theta^{1- \gamma} \ln(\frac{\eps}{\theta})  > 1-x^*.
    \end{align*}
    Using $-\theta^{1-\gamma} \ln(\frac{\eps}{\theta})=o(\theta^{-\alpha})$, for a given couple $(\alpha,\eps)$, there exists $C_{\eps,\alpha}$ such that 
    \begin{align*}
        x^* \geq 1-C_{\eps,\alpha}\theta^{-\alpha} .
    \end{align*}
\end{proof}

\section{Hyperbolic stage, first part: toward a neighborhood of a steady state}

\subsection{The control strategy}
The goal of this section is to prove the following proposition.
\begin{proposition}
    \label{prop:main_first_stage}
    For a given $y_0\in L^\infty(0,1)$, $\eta>0$ and $T>0$, one can find $\theta_0>0$ such that the following holds: for any $\theta \geq \theta_0$, there exist $u,v\in L^\infty(0,T)\times (L^\infty(0,T)\cap H^{1/4}(0,T))$ such that the solution $y$ of \eqref{eq:gen_burgers} satisfies 
    \begin{align}
        \label{eq:lower_bound_step1}
        \vartheta(x) \leq y(T,x),
    \end{align}
    and
    \begin{align}
        \label{eq:upper_bound_step1}
        y(T,x) \leq \theta + \eta.
    \end{align}
\end{proposition}
For the proof of \cref{prop:main_first_stage}, we use the following controls on $[0,T]$
\begin{align}
    \label{eq:def_control}
    u(t)= \left\{ 
        \begin{array}{ll}
        \frac{\theta+2\|y_0\|_{L^\infty}}{T'}& \text{ for }t\leq T' ,\\
        0& \text{ for }t> T',
    \end{array}
    \right.
\end{align}
and
\begin{align}
    \label{eq:def_control2}
    v(t)= \left\{ \begin{array}{ll}
        \frac{(\theta+\|y_0\|_{L^\infty})t}{T'}& \text{ for }t\leq T' ,\\
        \theta +\frac{\|y_0\|_{L^\infty}(\frac{T}{2}-t)}{\frac{T}{2}-T'}& \text{ for }T'<t\leq T/2 ,\\
        \theta& \text{ for }t> T/2 ,
    \end{array}
    \right.
\end{align}
where $T'<\frac{T}{2}$ will be chosen small enough later.
We denote $y$ the solution of \eqref{eq:gen_burgers} associated with these controls. We prove separately the lower bound \eqref{eq:lower_bound_step1} and the upper bound \eqref{eq:upper_bound_step1}.

\begin{remark}
    In what follows, we use $C$ for a constant which is independent of $\theta$ and may be different from one line to the next.
\end{remark}
\subsection{Lower bound}
The idea to handle the lower bound is to use a very small time $T'$. Let us introduce 
\begin{align}
    \un v(t)= \frac{(\theta +2\|y_0\|_{L^\infty})t}{T'}-\|y_0\|_{L^\infty}  \leq v(t)& \text{ for all }t\leq T' ,
\end{align}
and set
\begin{align}
    T'_0=\frac{T'\|y_0\|_{L^\infty}}{(\theta +2\|y_0\|_{L^\infty})}.
\end{align}
Leading to, $\forall t \leq  T'_0,\, \un v(t)\leq 0$.

We define the following subsolution $\un{y}$ of $y$ on $(0, T'_0)$:
\begin{equation}
    \label{eq:sub_sol1}
    \left\{
        \begin{array}{ll}
            \un{y}_t+\gamma |\un{y}|^{\gamma-1} y_x-\un{y}_{xx}=u(t)& \text{on }(0, T'_0)\times (0,1) , \\
            \un{y}(t,0)=\un v(t)& \text{on }(0, T'_0), \\
            \un{y}(t,1)=\un v(t) & \text{on }(0, T'_0), \\
            \un{y}(0,x)=-\|y_0\|_{L^\infty}& \text{on }(0,1) .
        \end{array}
        \right.
\end{equation}
We easily check that $\un{y}(t)=\un v(t)$ for all $t\leq  T'_0$, which means in particular that $\un{y}( T'_0)=0$.

Let us now study the solution $\un{y}^{\text{lin}}$ of the following heat equation on the semi-infinite space domain $(-\infty,1)$ with Dirichlet boundary condition on $x=1$:

\begin{equation}
    \label{eq:sub_sol_lin}
    \left\{
        \begin{array}{ll}
            \un{y}^{\text{lin}}_t-\un{y}^{\text{lin}}_{xx}=u(t) & \text{on }( T'_0,T')\times (-\infty,1)  ,\\
            \lim_{x\to -\infty}\un{y}^{\text{lin}}(t,x)=\un v(t) & \text{on }( T'_0,T') ,\\
            \un{y}^{\text{lin}}(t,0)=0 & \text{on }( T'_0,T') ,\\
            \un{y}^{\text{lin}}(T'_0,x)=0& \text{on }(-\infty,1)  .
        \end{array}
        \right.
\end{equation}
Using the usual representation formula, we introduce $\erf(z)=\frac{2}{\sqrt{\pi}}\int_0^z e^{-x^2}dx$, and compute
\begin{align}
    \un{y}^{\text{lin}}(t,x_1) &= \int_{T'_0}^{t} \frac{1}{\sqrt{4 \pi (t-s)}}\int_{-\infty}^1 (e^{\frac{-(x_1-x_2)^2}{4(t-s)}}-e^{\frac{-(x_1+x_2-2)^2}{4(t-s)}}) u(s) dx_2 ds \notag \\
    &= \int_{T'_0}^{t} u(s)  \erf (\frac{1-x_1}{\sqrt{t-s}} )ds.\notag \\
    &= \int_{T'_0}^{t} \frac{\theta+2\|y_0\|_{L^\infty}}{T'}  \erf (\frac{1-x_1}{\sqrt{t-s}} )ds.
\end{align}
Note that for $t\in [T'_0,T']$,
\begin{align}
    \un{y}^{\text{lin}}_x(t,x) \leq 0,\quad \text{and} \quad 0\leq \un{y}^{\text{lin}}(t,x) \leq \un v(t).
\end{align}
Besides, using that $\erf$ is an increasing function, 
\begin{align}
     (\theta+\|y_0\|_{L^\infty} )  \erf (\frac{1-x}{\sqrt{T'}} ) \leq \un{y}^{\text{lin}}(T',x).
\end{align}
As $\vartheta$ is a concave function, it lies below its tangent at $x=0$ and $x=1$:
\begin{align}
    \vartheta(x)\leq \min \left[ \theta, \vartheta_x(1) (x-1) \right], \qquad x\in[0,1].
\end{align}
Using that $\lim_{z\to \infty} \erf(z)=1$, one can prove that for $T'$ small enough, 
\begin{align}
    \min \left[ \theta, \vartheta_x(1) (x-1) \right] \leq (\theta+\|y_0\|_{L^\infty} )  \erf (\frac{1-x}{\sqrt{T'}} ), \qquad x\in[0,1].
\end{align}
Combining the last three inequalities, we have for $T'$ small enough that
\begin{align}
    \vartheta(x) \leq \un{y}^{\text{lin}}(T',x),& \quad \text{for all } x\in[0,1].
\end{align}
Going back to the non-linear equation, we extend $\un{y}$ to $(T'_0,T')$ by
\begin{equation}
        \label{eq:sub_sol2}
        \left\{
        \begin{array}{ll}
            \un{y}_t+(\un{y}^\gamma)_x-\un{y}_{xx}=u(t)& \text{on }(T'_0,T')\times (0,1)  , \\
            \un{y}(t,0)=\un{y}^{\text{lin}}(t,0)& \text{on }(T'_0,T') , \\
            \un{y}(t,1)=0 & \text{on }(T'_0,T') , \\
            \un{y}(T'_0,x)=0& \text{on }(0,1) .
        \end{array}
        \right.
\end{equation}
Note that $\un{y}$ is non-negative on $(T'_0,T')\times (0,1)$.
We consider $\delta= \un{y}-\un{y}^{\text{lin}}$, which is solution of
\begin{equation}
    \label{eq:delta_sub_sol}
        \left\{
        \begin{array}{ll}
            \delta_t-\delta_{xx}=-((\un{y}^{\text{lin}}+\delta)^\gamma)_x& \text{on }(T'_0,T')\times (0,1)  , \\
            \delta(t,0)=0& \text{on }(T'_0,T') , \\
            \delta(t,1)=0 & \text{on }(T'_0,T') , \\
            \delta(T'_0,x)=0& \text{on }(0,1) .
        \end{array}
        \right.
\end{equation}
Let us prove that $\delta \geq 0$ thanks to an energy estimate. We multiply \cref{eq:delta_sub_sol} by $w=\min(\delta,0)$, the negative part of $\delta$, and we integrate in space to get
\begin{align*}
    \frac{1}{2}\frac{d}{dt} \|w(t)\|^2_{L^2}+ \| w_x(t)\|^2_{L^2}
    &=-\int_0^1 w((\un{y}^{\text{lin}}+w)^\gamma)_x \\
    &=-\int_0^1 w ((\un{y}^{\text{lin}}+w )^\gamma-(\un{y}^{\text{lin}})^\gamma +(\un{y}^{\text{lin}})^\gamma )_x\\
    &=-\int_0^1 \gamma w \un{y}^{\text{lin}}_x (\un{y}^{\text{lin}})^{\gamma-1} + \int_0^1 w_x ((\un{y}^{\text{lin}}+w )^\gamma-(\un{y}^{\text{lin}})^\gamma).
\end{align*}
The first term of the last line is non-positive, hence using Cauchy-Schwarz and Young inequality, we obtain
\begin{align*}
    \frac{1}{2}\frac{d}{dt} \|w(t)\|^2_{L^2}+ \| w_x(t)\|^2_{L^2}&\leq \|w_x(t)\|_{L^2} \|(\un{y}^{\text{lin}}+w )^\gamma-(\un{y}^{\text{lin}})^\gamma\|_{L^2} \\
    & \leq \frac{1}{2} \|w_x(t)\|_{L^2}^2 + \frac{\gamma^2}{2} (\theta+\|y_0\|_{L^\infty})^{2(\gamma-1)} \|w(t)\|^2_{L^2}.
\end{align*}
By applying Gronwall inequality, we get
\begin{align}
    \|w(t)\|_{L^2}^2 \leq \|w(T'_0)\|_{L^2}^2 e^{\gamma^2 (\theta+\|y_0\|_{L^\infty})^{2(\gamma-1)} (t-T'_0)},
\end{align}
which implies $w=0$.

Hence, we proved that for $T'$ small enough, $y(T')\geq \vartheta$. Besides, $\vartheta$ is a subsolution for the controls defined in \cref{eq:def_control,eq:def_control2} on $(T',T)$. As a consequence, for all $t\geq T'$, $
y(t,x)\geq \vartheta(x)$, which concludes the proof of the lower bound \eqref{eq:lower_bound_step1} of \cref{prop:main_first_stage}.
\subsection{Upper bound}
We set
\begin{align}
\bar v(t)= \left\{ \begin{array}{lll}
    \|y_0\|_{L^\infty} + \frac{(\theta+2\|y_0\|_{L^\infty})t}{T'} &\geq v(t)& \text{ for }t\leq T' ,\\
    v(t) & & \text{ for }t> T' .
\end{array}
\right. 
\end{align}
Let us consider the supersolution $\bar{y}$ of $y$ on $(0,T)$\footnote{As $\bar{v}$ is not in $H^{1/4}(0,T)$, we formaly define $\bar{y}$ on both time intervals $(0,T')$ and $(T',T)$.}
\begin{equation}
    \label{eq:super_sol1}
    \left\{
        \begin{array}{ll}
            \bar{y}_t+(\bar{y}^\gamma)_x-\bar{y}_{xx}=u(t)& \text{on }(0,T)\times (0,1)  , \\
            \bar{y}(t,0)=\bar v(t)& \text{on }(0,T) , \\
            \bar{y}(t,1)=\bar v(t) & \text{on }(0,T) , \\
            \bar{y}(0,x)=\|y_0\|_\infty& \text{on }(0,1) .
        \end{array}
        \right.
\end{equation}
Note that $\forall t \in [0,T'],\, \bar{y}(t)=\bar v(t)$.
Using the comparison principle, we obtain
\begin{align}
    \label{eq:apriori_supersol_L_infty_bound}
    \theta \leq \bar{y}(t)\leq \theta +3 \|y_0\|_{L^\infty}& \text{ for all }t \geq T'.
\end{align}
Let us now denote by $\delta(t,x)$ the solution of
\begin{equation}
    \label{eq:non_linear_diff_new}
    \left\{
    \begin{array}{ll}
        \delta_t+ \gamma \theta^{\gamma-1} \delta_x-\delta_{xx}=-\big( (\theta+\delta)^\gamma -\theta^\gamma- \gamma \theta^{\gamma-1} \delta\big)_x & \text{on }(\frac{T}{2},T)\times (0,1)  , \\
        \delta(t,0)=0& \text{on }(\frac{T}{2},T) , \\ 
        \delta(t,1)=0& \text{on }(\frac{T}{2},T) , \\
        \delta(\frac{T}{2},x)=3 \|y_0\|_{L^\infty}& \text{on }(0,1) .
    \end{array}
    \right.
\end{equation}
Using the comparison principle, we observe that, for $t\in [\frac{T}{2},T]$, $\delta(t,x)\geq \bar y(t,x)-\theta$.
To study the evolution of $\delta$, we introduce the weight
\begin{align}
    \label{eq:def_A}
    A(x)=e^{\frac{\gamma}{2} \theta^{\gamma-1}(1-x)}.
\end{align}
Thus,
\begin{align}
     A_x=-\frac{\gamma}{2} \theta^{\gamma-1} A.
\end{align}
We multiply the first line of \eqref{eq:non_linear_diff_new} by $A \delta $ and we integrate on space. The terms of the left-hand side are
\begin{align}
    \int_0^1 \delta _t A \delta  dx&=\frac{1}{2}\frac{d}{dt} \|\delta \|^2_{L^2(Adx)},\\
    \int_0^1 \gamma \theta^{\gamma-1} \delta_x  A \delta dx &= \int_0^1 \left( \gamma \theta^{\gamma-1} \right)^2 \delta^2A, \\
    -\int_0^1 \delta _{xx} A\delta  dx&= \int_0^1 (\delta _x)^2 Adx - \int_0^1 \frac{\gamma}{2} \theta^{\gamma-1} \delta \delta_x A dx.
\end{align}
Hence, the left-hand side of Equation \eqref{eq:non_linear_diff_new} becomes
\begin{align}
    \label{eq:linear_enery_estimate}
    \frac{1}{2}\frac{d}{dt} \|\delta \|^2_{L^2(Adx)} +\frac{\gamma^2 \theta^{2(\gamma-1)}}{2} \|\delta \|^2_{L^2(Adx)} + \|\delta _x\|^2_{L^2(Adx)}.
\end{align}
We multiply the right-hand side of \cref{eq:non_linear_diff_new} by $A\delta$ and integrate in space to get
\begin{align}
    -\int_0^1 A \delta \big( (\theta+\delta)^\gamma - \gamma \theta^{\gamma-1} \delta\big)_x
    \label{eq:non_lin_with_ax}
    =\int_0^1 A_x \delta \big( (\theta+\delta)^\gamma -\theta^\gamma- \gamma \theta^{\gamma-1} \delta \big)\\
    +\int_0^1 A \delta_x \big( (\theta+\delta)^\gamma -\theta^\gamma- \gamma \theta^{\gamma-1} \delta \big).
\end{align}
As $z\mapsto z^\gamma$ is convex, $(\theta+\delta)^\gamma -\theta^\gamma- \gamma \theta^{\gamma-1} \delta \geq 0$. As a consequence, the right-hand side integral in line \eqref{eq:non_lin_with_ax} is non-positive.
Besides,

\begin{align}
    \label{eq:inequality_non_linear_right}
    \int_0^1 A \delta_x \big( (\theta+\delta)^\gamma - \gamma \theta^{\gamma-1} \delta \big) 
    & \leq \frac{1}{2}\|\delta_x\|^2_{L^2(Adx)} + \frac{1}{2}\| (\theta+\delta)^\gamma -\theta^\gamma- \gamma \theta^{\gamma-1} \delta\|^2_{L^2(Adx)}.
\end{align}
Let us estimate the second term of the right-hand side.
For $\gamma \geq 2$, we can use the fact that the second derivative of $z\mapsto z^\gamma$ is increasing to get
\begin{align}
    \label{eq:inequality_non_linear_right_big_gamma}
    (\theta+\delta)^\gamma -\theta^\gamma- \gamma \theta^{\gamma-1} \delta \leq \frac{\delta^2}{2} \gamma(\gamma-1)(\theta+3 \|y_0\|_{L^\infty} )^{\gamma-2}.
\end{align}
Whereas, for $\gamma\leq 2$, we get
\begin{align}
    \label{eq:inequality_non_linear_right_small_gamma}
    (\theta+\delta)^\gamma -\theta^\gamma- \gamma \theta^{\gamma-1} \delta \leq \frac{\delta^2}{2} \gamma(\gamma-1)\theta^{\gamma-2}.
\end{align}
Combining \cref{eq:linear_enery_estimate,eq:inequality_non_linear_right,eq:inequality_non_linear_right_big_gamma,eq:inequality_non_linear_right_small_gamma}, we obtain the weighted energy estimate

\begin{align}
    \label{eq:Lyapunov_upper}
    \frac{1}{2} \frac{d}{dt}\|\delta(t) \|^2_{L^2(Adx)}
    +\frac{1}{2} \|\delta_x \|_{L^2(Adx)}^2 + \frac{\gamma^2 \theta^{2(\gamma-1)}}{2} \|\delta(t) \|^2_{L^2(Adx)}
    \leq 
    \gamma^2 (\gamma -1)^2 \tilde \theta^{2(\gamma-2)} \| \delta(t)\|_{L^\infty}^2
     \|\delta(t) \|^2_{L^2(Adx)},
\end{align}
where 
\begin{align}
    \tilde \theta = 
    \left\{
    \begin{array}{lr}
        \theta  & \text{if } 3/2<\gamma < 2 , \\
        \theta + 3 \|y_0\|_{L^\infty}& \text{if } \gamma \geq 2 .
    \end{array}
    \right.
\end{align}
We obtain an $L^\infty$ norm estimate for $\delta$ with \cref{eq:apriori_supersol_L_infty_bound}
\begin{align}
    \| \delta(t)\|_{L^\infty}^2\leq 9 \|y_0\|_{L^\infty}^2.
\end{align}
Hence, there exists $\theta_0$ (depending on $\|y_0\|_{L^\infty}$ and $\gamma$) which together with \cref{eq:Lyapunov_upper} implies that for any $\theta \geq \theta_0$

\begin{align}
    \label{eq:Lyapunov_upper1}
    \frac{1}{2} \frac{d}{dt}\|\delta(t) \|^2_{L^2(Adx)}
    +\frac{1}{2} \|\delta_x \|_{L^2(Adx)}^2 \leq -\frac{\gamma^2 \theta^{2(\gamma-1)}}{4} \|\delta(t) \|^2_{L^2(Adx)}.
\end{align}
Using the expression of $A$ \eqref{eq:def_A}, we estimate the initial condition
\begin{align}
    \label{eq:L2A_estimate}
    \|\delta(\frac{T}{2})\|_{L^2(Adx)}^2\leq \|\delta(\frac{T}{2})^2 A\|_{L^\infty} \leq 9 \|y_0\|_{L^\infty}^2 e^{\frac{\gamma}{2} \theta^{\gamma-1}}.
\end{align}
Gronwall inequality applied to \cref{eq:Lyapunov_upper1,eq:L2A_estimate} gives
\begin{align}
    \|\delta(t) \|^2_{L^2(Adx)}\leq  9\|y_0\|_{L^\infty}^2 e^{\frac{\gamma}{2} \theta^{\gamma-1}} e^{-\frac{\gamma^2 \theta^{2(\gamma-1)}}{2} (t-T')}.
\end{align}
Integrating \cref{eq:Lyapunov_upper1} on $(\frac{3T}{4},T)$, we get
\begin{align}
    \|\delta_x \|^2_{L^2(\frac{3T}{4},T;L^2(Adx))}&\leq  \|\delta(3T/4)\|_{L^2(Adx)} \notag \\
    \label{eq:lyapunov_exp_decrease}
    &\leq  9\|y_0\|_{L^\infty}^2 e^{\frac{\gamma}{2} \theta^{\gamma-1}} e^{-\frac{\gamma^2 \theta^{2(\gamma-1)}}{2} (T/4)}
    \xrightarrow{\theta \to \infty} 0.
\end{align}
Hence, there exists $t^*\in [\frac{3T}{4},T]$ such that 
\begin{align}
    \|\delta(t^*)\|^2_{H^1}\leq C \|\delta_x(t^*)\|^2_{L^2(Adx)}\leq \frac{4C}{T} \|\delta_x \|^2_{L^2(\frac{3T}{4},T;L^2(Adx))}.
\end{align}
Using a classical parabolic estimate on the heat equation (see \cref{ap:el_reg}), we get
\begin{align}
    \|\delta(T)\|_{H^1_0} &\leq \|\delta(t^*)\|_{H^1_0}
    + \| \big((\theta+\delta)^\gamma - \theta^\gamma\big)_x\|_{L^2(t^*,T;L^2(dx))}^2\\
    &\leq \frac{4C}{T} \|\delta_x \|^2_{L^2(\frac{3T}{4},T;L^2(Adx))}
    +\gamma^2 (\theta+3\|y_0\|_{L^\infty})^{2(\gamma-1)}  \|\delta\|^2_{L^2(t^*,T;H^1_0)} .
\end{align}
Thanks to the exponential decrease in $\theta$ of \cref{eq:lyapunov_exp_decrease}, we obtain that 
\begin{align}
    \|\delta(T)\|_{H^1_0} \xrightarrow{\theta\to \infty} 0,
\end{align}
which proves \cref{eq:upper_bound_step1} and conclude the proof of \cref{prop:main_first_stage}.

\section{Hyperbolic stage, second part: toward a neighborhood of zero up to a boundary layer}
Thanks to the previous part, we reduced our problem to the case where the initial condition $y_0$ satisfies for some $\theta, \eta>0$,
\begin{align}
    \label{eq:initial_cond}
    \vartheta(x) \leq y_0(x) \leq \theta+ \eta.
\end{align}
In this section we prove that we can steer the solution of the system \eqref{eq:gen_burgers} to a small neighborhood of the null state up to a boundary residue near the right endpoint.

\begin{lemma}
    \label{lem:neighborhood_of_zero}
    Let $T>0$ and $\theta>0$. There exists $T'\leq T$, $u,v\in L^\infty(0,T') \times (L^\infty(0,T')\cap H^{1/4}(0,T'))$ such that
    for any $y_0$ satisfying \cref{eq:initial_cond}, we have 
    \begin{align}
        \vartheta(x)-\theta-\eta<y(T',x)<\eta.
    \end{align}
\end{lemma}
    \begin{proof}
    Let us consider the controls $u(t)= -\frac{\theta}{T'}$ and $v(t)=\theta (1- \frac{t}{T'})$ on $[0,T']$ for some $T'\leq T$ which will be chosen later. We denote by $y$ the corresponding solution of \cref{eq:gen_burgers} with any initial condition satisfying \cref{eq:initial_cond}. Then, we define a subsolution $\un{y}$ and a supersolution $\bar{y}$ by
    \begin{equation}
        \left\{
            \begin{array}{ll}
                \un{y}_t+\gamma|\un{y}|^{\gamma-1}\un{y}_x-\un{y}_{xx}=u(t)& \text{on }(0,T')\times (0,1)  , \\
                \un{y}(t,0)=v(t)& \text{on }(0,T') , \\
                \un{y}(t,1)=-\frac{\theta t}{T'} & \text{on }(0,T') , \\
                \un{y}(0,x)=\vartheta& \text{on }(0,1) ,
            \end{array}
            \right.
    \end{equation}
    and 
    \begin{equation}
        \left\{
            \begin{array}{ll}
                \bar{y}_t+\gamma|\bar{y}|^{\gamma-1}\bar{y}_x-\bar{y}_{xx}=u(t)& \text{on }(0,T')\times (0,1)  ,\\
                \bar{y}(t,0)=v(t)+\eta& \text{on }(0,T') ,\\
                \bar{y}(t,1)=v(t)+\eta & \text{on }(0,T') ,\\
                \bar{y}(0,x)=\theta+\eta& \text{on }(0,1) .
            \end{array}
        \right.
    \end{equation}
    We easily check that $\bar{y}(t)=v(t)+\eta$, thus $\bar{y}(T')=\eta$. This concludes the proof of the upper bound.

    Let us now focus on the subsolution. We define $\delta(t,x)=\un{y}(t,x)-\vartheta(x)+\frac{\theta t}{T'}$. Then, $\delta$ is solution of
    \begin{equation}
            \left\{
            \begin{array}{ll}
                \delta_t-\delta_{xx}=(\vartheta^\gamma)_x-\gamma |\vartheta-\frac{\theta t}{T'}+\delta|^{\gamma-1} (\vartheta-\frac{\theta t}{T'}+\delta)_x& \text{on }(0,T')\times (0,1)  , \\
                \delta(t,0)=0& \text{on }(0,T') , \\
                \delta(t,1)=0 & \text{on }(0,T') , \\
                \delta(0,x)=0& \text{on }(0,1) .
            \end{array}
            \right.
    \end{equation}
    First, note that we have the bound 
    \begin{align*}
        \|(\vartheta^\gamma)_x\|_{H^{-1}}\leq C(\theta,\gamma )
    \end{align*}
    and notice that 
    \begin{align*}
        \|\left(\gamma |\vartheta-\frac{\theta t}{T'}+\delta|^{\gamma-1} (\vartheta-\frac{\theta t}{T'}+\delta)_x \right)\|_{H^{-1}}=\|\left( \left| \vartheta-\frac{\theta t}{T'}+\delta(t) \right|^\gamma \right)_x\|_{H^{-1}}.
    \end{align*}
    We can estimate this $H^{-1}(0,1)$ norm as follows:
    \begin{align*}
        \|\left( \left| \vartheta-\frac{\theta t}{T'}+\delta(t) \right|^\gamma \right)_x\|_{H^{-1}}
        &=\|\left( \left| \vartheta-\frac{\theta t}{T'}+\delta(t) \right|^\gamma \right)_x-\left( \left| \vartheta-\frac{\theta t}{T'} \right|^\gamma \right)_x+\left( \left| \vartheta-\frac{\theta t}{T'} \right|^\gamma \right)_x\|_{H^{-1}}\\
        &\leq \|\left( \left| \vartheta-\frac{\theta t}{T'}+\delta(t) \right|^\gamma \right)- \left( \left| \vartheta-\frac{\theta t}{T'} \right|^\gamma \right)\|_{L^2}+C(\theta,\gamma )\\
        &\leq C(\theta,\gamma ) \|\delta(t)\|_{L^2}+C(\theta,\gamma ).
    \end{align*}
    Thus, applying a regularity estimate for the heat equation reproduced in \cref{ap:el_reg},
    \begin{align}
        \|\delta(t)\|^2_{L^2} +\|\delta\|_{L^2(0,t;H^1_0)}^2 \leq \int_0^t C(\theta,\gamma ) (\|\delta(s)\|^2_{L^2}+1) ds.
    \end{align}
    Thanks to Gronwall inequality,
    \begin{align}
        \|\delta(t)\|_{L^2}^2 \leq C(\theta,\gamma )t e^{C(\theta,\gamma )t}.
    \end{align}
    After an integration in time, we obtain
    \begin{align}
        \|\delta\|_{L^2(0,t;H^1_0)}^2 \leq \frac{1}{C(\theta,\gamma )} \left(e^{C(\theta,\gamma )t} (C(\theta,\gamma )t-1) +1 \right)= O_{t\to 0}(t^2).
    \end{align}
    Now, using that there exists a constant $C(\theta,\gamma)$ such that
    \begin{align}
        \|(|\vartheta+\frac{\theta t}{T}+\delta(t)|^\gamma)_x\|_{L^2}\leq C(\theta,\gamma ) \|\delta(t)\|_{H^1} +C(\theta,\gamma ),
    \end{align}
    together with \cref{ap:el_reg} again, we get
    \begin{align}
        \|\delta(t)\|^2_{H^1_0}\leq C(\theta,\gamma ) \|\delta\|_{L^2(0,t;H^1_0)}^2 +C(\theta,\gamma )t.
    \end{align}
    As those estimates are uniform in $T'$, we have
    \begin{align}
        \|\delta(T')\|^2_{H^1_0} =O_{T'\to 0}(T').
    \end{align}
    This implies that there exists $T'<T$ such that
    \begin{align}
        \|\delta(T')\|_{L^\infty}<\eta.
    \end{align}
    This concludes the proof of \cref{lem:neighborhood_of_zero}.
\end{proof}
\section{Passive stage: dissipation of the boundary residue}
\label{sec:passive_stage}
In this stage, we start with an initial condition satisfying
\begin{align}
    \label{eq:init_stage_2}
        \vartheta(x)-\theta-\eta<y(T',x)<\eta, \qquad x\in[0,1]
\end{align}
and we prove the dissipation of the residue $ \vartheta(x)-\theta$.
\begin{proposition}
    \label{prop:main_stage2}
    For a given $T>0$, there exist some constants $C(T)>0$ and $\theta_0>0$ such that for every $\theta\geq \theta_0$, every $\eta>0$ in a neighborhood of zero and every initial condition $y_0$ satisfying \cref{eq:init_stage_2}, the solution of \cref{eq:gen_burgers} with null controls satisfies 
    \begin{align}
       -C(T) \eta \leq y(T)\leq \eta, \quad & y(T)\in H^1_0(0,1).
    \end{align}
\end{proposition}
The upper bound is just a consequence of the comparison principle. Let us thus focus on the lower bound. If $\gamma$ is smaller, the boundary layer is larger and the time needed for the dissipation increases. \cref{fig:dissipation} illustrates this phenomenon. If $\gamma\leq 3/2$, our estimates are not sufficient to prove \cref{prop:main_stage2}.
\begin{figure}
    \includegraphics[width=0.85\textwidth]{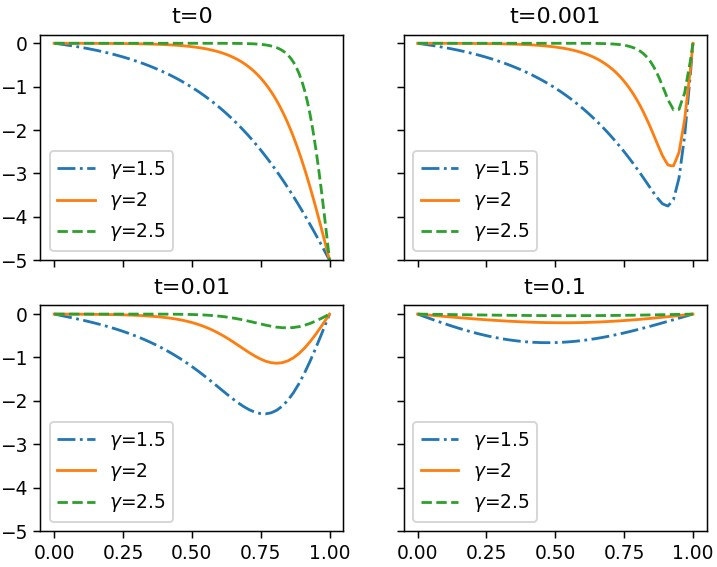}
    \caption{Simulation of the dissipation of the boundary layer residue for different values $\gamma$.}
    \label{fig:dissipation}
\end{figure}
In order to estimate the size of the boundary layer, we use the second point of \cref{lem:boundary_layer}. That is, for $\frac{1}{2}<\alpha< \gamma -1$, there exists $\tilde C$ such that 
\begin{align}
    \label{eq:initial_estimate_uny}
     - 2\eta -\theta  \mathbb{1}_{x\geq 1-\tilde C \theta ^{-\alpha}} \leq  \vartheta(x)-\theta -\eta.
\end{align}
We also set $\eps_1=\gamma-1-\alpha>0$ as we will need later to take $\alpha$ close enough to $\gamma-1$.

We introduce the following subsolution:
\begin{align}
    \label{eq:boundary_inf}
    \left\{
        \begin{array}{rll}
             \un y_t +( |\un y|^\gamma)_x- \un y_{xx}&=0 & \text{on }(0,T)\times (0,1)  , \\
             \un y(t,0)&=0 & \text{on }(0,T) , \\
             \un y(t,1)&=0 & \text{on }(0,T) , \\
             \un y(0)&=- 2\eta -\theta  \mathbb{1}_{x\geq 1-\tilde C \theta ^{-\alpha}} & \text{on }(0,1) .
        \end{array}
    \right.
\end{align}
Note that $\un y$ is a non-positive solution. Besides
\begin{align}
    \| \un y(0) \|_{L^1}=2 \eta +\tilde C \theta^{2-\gamma+\eps_1}.
\end{align}
Hence, if $\gamma >2$, we have
\begin{align}
    \lim_{\theta\to \infty}  \| \un y(0) \|_{L^1}=2 \eta.
\end{align}
We use a smoothing result by Carlen and Loss \cite[Theorem 1]{carlenOptimalSmoothingDecay1995} for strictly convex conservation laws to deduce
\begin{align}
    \label{eq:smoothing_CL}
    \| \un y(t) \|_{L^\infty} \leq \frac{\| \un y(0) \|_{L^1}}{\sqrt{4\pi t}}.
\end{align}
As a consequence, there exists $\theta_0$ such that, for any $\theta \geq \theta_0$, 
\begin{align}
    \| \un y(T) \|_{L^\infty}\leq \frac{3 \eta}{\sqrt{4 \pi T}},
\end{align}
which concludes the proof for $\gamma>2$.

In the case  $\gamma <2$, the boundary layer grows with $\theta $. The smoothing property of \cref{eq:smoothing_CL} does not exploit the initial localisation near the boundary. Hence, to handle $\gamma$ small, we use instead a weighted estimate.


We multiply \cref{eq:boundary_inf} by $(x-1)$ and integrate on space to get
\begin{align}
    \label{eq:estimate_weighted}
    \frac{d}{dt}\int_{0}^1 (x-1)\un y (t,x) dx + \int_{0}^1 |\un y(t,x)|^\gamma dx -\un y_x(0)=0.
\end{align}
Note that $\int_{0}^1 (x-1)\un y (t,x) dx$ and $-\un y_x(0)$ are both non-negative as $\un y \leq 0$. Besides,
\begin{align}
    \label{eq:init_weighted}
    \int_{0}^1 (x-1) (-2 \eta \mathbb{1}_{x\geq 0} -\theta  \mathbb{1}_{x\geq 1-\tilde C \theta ^{-\alpha}}) dx=\frac{\tilde C^2 \theta ^{1-2\alpha}}{2} + \eta.
\end{align}
We integrate \cref{eq:estimate_weighted} on $(0,\frac{T}{2})$ and use \cref{eq:init_weighted} to get the estimate
\begin{align}
    \label{eq:estimate_Lgamma}
    \|\un y\|^\gamma_{ L^\gamma((0,\frac{T}{2})\times (0,1)) } \leq \frac{\tilde C^2 \theta ^{1-2\alpha}}{2} + \eta.
\end{align}
Note that $\theta ^{1-2\alpha}$ goes to $0$ as $\theta$ goes to $\infty$. This is the case because we assumed $\gamma >3/2$.

Using \cref{eq:estimate_Lgamma}, there exists $C$ such that for $\theta$ large enough, there exist $t^*\in (0,\frac{T}{2})$ satisfying
\begin{align}
    \|\un y(t^*)\|_{L^1(0,1)}< C \eta.
\end{align}
We cannot use the smoothing result from Carlen and Loss as it requires the flux function to be strongly convex (hence $\gamma\geq 2$), but a similar result (without an explicit expression for the constant) is available in \cite[Lemma 3.1]{feireislL1StabilityConstant1999} or \cite[Lemma 4.2]{bankViscousConservationLaws2020}. Finally, we get
\begin{align}
    \|\un y (T)\|_{L^\infty} \leq \frac{C}{\sqrt{T-t^*}}\|\un y(t^*)\|_{L^1} \leq C' \eta.
\end{align}
This concludes the proof of \cref{prop:main_stage2}.
\section{Parabolic stage: local null exact controllability}
\label{subsec:local_exact}

The local null exact controllability of semi-linear parabolic equations has been established for a wide variety of cases. For completeness, we provide a sketch of proof based on \cite[Lemma 2]{fernandezNullControllabilityBurgers2007}.

\begin{lemma}
    \label{lem:local_exact}
    There exists $C^*>0$ such that for every $  y_0  \in {H^1_0}$ satisfying $\|y_0\|_{L^\infty}<\frac{1}{2\gamma}$ and $\|y_0\|_{L^2}\leq \frac{1}{2\gamma T} e^{-C^*/T}$, there exists a control $v\in H^{1/4}(0,T)\cap L^\infty(0,T)$ such that the solution of \eqref{eq:gen_burgers} with controls $u=0$ and $v$ satisfies $y(T)=0$.
\end{lemma}

\begin{proof}[Sketch of the proof]
    First, we transform the boundary control problem into an internal control one.
    Namely, we consider the wider space domain $[-1,1]$ and a domain of the internal control $\omega\subset (-1,0)$ with non-empty interior. We denote $Q=(0,T)\times [-1,1]$. Let $w$ be an internal control acting on $(0,T)\times \omega$. We are interested in the local null exact controllability of
    \begin{equation}
        \label{eq:gen_burgers_extended}
        \left\{
            \begin{array}{ll}
                y_t+\gamma |y|^{\gamma-1}y_x-y_{xx}=w(t,x)\mathbb{1}_{(0,T)\times \omega} \quad & \text{on }Q  , \\
                y(t,0)=0& \text{on }(0,T) , \\
                y(t,1)=0& \text{on }(0,T) , \\
                y(0,x)=\tilde y_0(x)&\text{on }(-1,1) ,
            \end{array}
        \right.
    \end{equation}
    where $\tilde y_0\in H^1_0(0,1)$ is the extension of $y_0$ by $0$ on $(-1,1)$.
    
    We set $s\in (0,1)$ and introduce the closed convex set
    \begin{align}
        K= \{ z \in H^s(Q) \mid \|z\|_{L^\infty(Q)}\leq 1/\gamma\}
    \end{align}
    and the following set of admissible controls
    \begin{align}
        \mathcal{A}_0= \{ w \in L^\infty((0,T)\times \omega) \mid \|w\|_{L^\infty(Q)}\leq \|y_0\|_{L^2} e^{C^* /T}\},
    \end{align}
    where $C^*$ will be defined later.
    Let us define $\varphi$ for $z\in K$ by
    \begin{align}
        \varphi(z)= \gamma |z|^{\gamma-1}
    \end{align}
    Note that $\|\varphi(z)\|_{L^\infty}\leq 1$.
    Let $\mathcal{A}: H^s(Q)\to H^s(Q)$ be the set-value mapping associating with $z\in H^s(Q)$ the set of solutions $y$ of the linear equation
    \begin{equation}
        \label{eq:burgers_linearized}
        \left\{
            \begin{array}{ll}
                y_t+\varphi(z) y_x-y_{xx}=w(t,x) \quad & \text{on }(0,T)\times (-1,1) , \\
                y(t,0)=0& \text{on }(0,T), \\
                y(t,1)=0& \text{on }(0,T), \\
                y(0,x)=\tilde y_0(x)&\text{on }(-1,1) ,
            \end{array}
        \right.
    \end{equation}
    with $w \in \mathcal{A}_0$ and satisfying $y(T)=0$.

    According to \cite{fernandezNullControllabilityBurgers2007,fernandezNullApproximateControllability2000}, there exists $C^*>0$ such that for any function $\varphi(z)$ with $\|\varphi(z)\|_{L^\infty}\leq 1$, there exists $w\in \mathcal{A}_0$ steering the solution of \cref{eq:burgers_linearized} to zero, i.e., $\mathcal{A}(z)$ is not empty.

    Let us now assume that $y_0$ satisfies the assumption of \cref{lem:local_exact} with $C^*$. We want to check that $\mathcal{A}$ satisfies the hypotheses of Kakutani's fixed point theorem (\cite[Chapter 2]{granasFixedPointTheory2003}).

    Using \cite[Chapter 3]{ladyzenskajaLinearQuasilinearEquations1968}, the solutions $y$ of \cref{eq:burgers_linearized} belong to 
    $$X=L^2(0,T;H^2(-1,1))\cap H^1(0,T;L^2(-1,1))$$
     and thus to $H^1(Q)$. Moreover, the maximum principle implies 
    \begin{align}
        \|y\|_{L^\infty(Q)}\leq \|y_0\|_{L^\infty(Q)}+T \|w\|_{L^\infty(Q)}\leq 1/ \gamma.
    \end{align}
    Hence, $\mathcal{A}$ maps $K$ into $K$ and, for any $z\in K$, $\mathcal{A}(z)$ is a non-empty convex compact subset of $H^s(Q)$.

    Let us now check that $\mathcal{A}$ is upper hemicontinuous on $K$.
    In pursuite of this goal, we set $\mu \in (H^s(Q))'$. We have to check that
    \begin{align}
        z \mapsto \sup_{y\in \mathcal{A}(z)} \langle \mu, y \rangle
    \end{align}
    is upper semi-continuous.
    
    Let $(z_n)_n\in K^\N$ be a converging sequence toward $z_\infty$ in $H^s(Q)$.
    Let us first observe that $\varphi : K \to L^2(Q)$ is continuous. Indeed, for $\gamma \geq 2$ we use the uniform boundedness of $K$
    \begin{align*}
        \|\varphi(z_a)-\varphi(z_b)\|_{L^2(Q)}\leq C(\gamma) \|z_a-z_b\|_{L^2(Q)}.
    \end{align*}
    On the other-hand, for $1<\gamma<2$, $x\mapsto x^{\gamma-1}$ is $\gamma-1$ H\"older continuous. As a consequence, it is also the case for $\varphi$:
    \begin{align*}
        |\varphi(z_a)-\varphi(z_b)\|_{L^2(Q)}\leq C(\gamma) \|(z_a-z_b)^{\gamma-1}\|_{L^2(Q)}\leq C(\gamma) (2T)^{\frac{2-\gamma}{2}}\|z_a-z_b\|_{L^2(Q)}^{\gamma-1}.
    \end{align*}
    Hence, $\varphi(z_n)\to \varphi(z_\infty)$ in $L^2(Q)$.

    By compactness of $\mathcal{A}(z_n)$, there exists $y_n \in \mathcal{A}(z_n)$ such that
    \begin{align}
        \sup_{y\in \mathcal{A}(z_n)} \langle \mu, y \rangle=\langle \mu, y_n \rangle.
    \end{align}
    Using energy estimates (\cite[Chapter 3]{ladyzenskajaLinearQuasilinearEquations1968}), $(y_n)_n$ is uniformly bounded in $X$. Thanks to Aubin-Lions Lemma \cite{aubinTheoremeCompacite1963}, up to a subsequence, there exists $y_{\infty}\in X$ such that
    \begin{align}
        y_{n,t} &\rightharpoonup y_{\infty,t} & \text{ weakly }L^2(Q),\\
        y_n &\to y_{\infty} & \left\{
        \begin{array}{r}
              \text{ weakly in }L^2(0,T;H^2(-1,1)),\\
         \text{ strongly in }L^2(0,T;H^1_0(-1,1)),\\
         \text{ in }\mathcal{C}^0(0,T;L^2(-1,1)).
        \end{array}
        \right.
    \end{align}
    As $\|\varphi(z_n)\|_{L^{\infty}(Q)}\leq 1$, up to subsequence, there exists $g\in L^\infty(Q)$ such that we have the weak star convergence 
    $\varphi(z_n) \overset{*}{\rightharpoonup} g$. By uniqueness of the limit in the sense of distribution, $g=\varphi(z_\infty)$. Thus, $\varphi(z_n)y_{n,x} \rightharpoonup \varphi(z_\infty) y_{\infty,x}$ in $L^2(Q)$.

    As a consequence, $y_\infty$ is solution of \cref{eq:burgers_linearized} with $\varphi(z_\infty)$ as coefficient. This shows that $y_\infty \in \mathcal{A}(z_\infty)$, i.e.,
    \begin{align}
        \limsup_{n\to \infty} \sup_{y\in \mathcal{A}(z_n)} \langle \mu, y \rangle \leq \sup_{y\in \mathcal{A}(z_\infty)} \langle \mu, y \rangle.
    \end{align}
    Hence, by Kakutani's fixed point theorem, there exists $\hat z\in K$ such that $\hat z \in \mathcal{A}(\hat z)$.

    Let us now take the trace of $\hat z$ on $(0,T)\times \{0\}$. As $\hat z \in X$, by \cite[Theorem 2.1]{MR0350178}, $\hat z(\cdot,0)\in H^{3/4}(0,T)$. This ensures that the restriction of $\hat z$ to $(0,T)\times(0,1)$ is an admissible trajectory of the system \eqref{eq:gen_burgers} and proves \cref{lem:local_exact}.
\end{proof}
\section{Conclusion and open problems}

\subsection{A few open problems}
\label{sec:ccl_and_open}
Let us now present some open problems in one space dimension related to this work.
\paragraph{} Does the small-time global null controllability hold for \cref{eq:gen_burgers} with $1<\gamma \leq 3/2$?
    Our entire control strategy is based on using the hyperbolicity of the evolution equation to dissipate the initial condition at the cost of a boundary layer. Then, for this non-small boundary layer to disappear, we use the fact that the moment $\int (1-x)|y|$ is small. A generalisation of this idea was used in \cite{coronSmalltimeGlobalExact2020} where the preparation of the dissipation of boundary layer plays a crucial role. In our case, the return method with $\vartheta$ ensures the dissipation in the case $\gamma > 3/2$ (the limiting step is \cref{sec:passive_stage}). In the linear case $\gamma=1$, there is no boundary layer. As a consequence, we believe that our method cannot be extended to $[1,3/2]$. Another approach would be to use some highly oscillating controls to ensure a better preparation of the boundary residue.

\paragraph{}    We can also ask if the small-time global null controllability holds for more general flux functions ($f(u)_x$ instead of $(|u|^\gamma)_x$). The extension of our proof to strictly convex viscous conservation laws should be possible. For more general functions, a precise study of the solutions of \cref{def:k_M} would be necessary.

\paragraph{} Another interesting direction would be to consider a dispersive model like the Korteweg–De Vries equation. With the help of two boundary controls and a uniform in space internal control, Chapouly proved in \cite{chapoulyGlobalControllabilityNonlinear2009} that the small-time global null controllability holds. To the best of the author knowledge, whether the small-time global null controllability holds without the use of the right Dirichlet boundary control (and possibly with the help of a right control on the derivative) remains an open question.

\subsection{Acknowledgement}
The author would like to thank Jean-Michel Coron for having attracted his attention to this control
problem and for extremely pertinent advice he provided.
\appendix
\section{Parabolic regularity estimates for the heat equation}
\label{ap:el_reg}
We recall here a well-known result on the regularity of the heat equation reproduced for example in \cite[Appendix 4.1]{leautaudUniformControllabilityScalar2012}.

Let $\mathcal{H}^m=D((-\Delta)^{m/2})$, $m\geq 0$, be the domain of the fractional Dirichlet Laplacian on $L^2(0,1)$, and $\mathcal{H}^{-m}$ the dual of $\mathcal{H}^m$ with pivot space $L^2(0,1)$. In particular $\mathcal{H}^1=H^1_0(0,1)$, $\mathcal{H}^0=L^2(0,1)$ and $\mathcal{H}^{-1}=H^{-1}(0,1)$.
Let $y$ be a classical solution of 

\begin{align}
\left\{
\begin{array}{ll}
     y_t-y_{xx}=f(t,x) & \text{on }(0,T)\times (0,1)  , \\
     y(t,0)=0 & \text{on }(0,T) , \\
     y(t,1)=0 & \text{on }(0,T) , \\
     y(0,\cdot)=y_0 & \text{on }(0,1) ,
\end{array}
\right.
\end{align}
with $m\in \R$, $u_0\in \mathcal{H}^m(0,1)$ and $f\in L^2(0,T;\mathcal{H}^{m-1}(0,1))$.
Then,
\begin{align}
y\in \mathscr{C}^0(0,T;\mathcal{H}^m(0,1))\cap L^2(0,T;\mathcal{H}^{m+1}(0,1))\cap H^1(0,T;\mathcal{H}^{m-1}(0,1)),
\end{align}
and, for $t\leq T$,
\begin{align}
\|y(t)\|_{\mathcal{H}^m}^2+\int_0^t \|y(s)\|^2_{\mathcal{H}^{m+1}}ds+\int_0^t \|y_t(s)\|^2_{\mathcal{H}^{m-1}}ds= \|y_0\|^2_{\mathcal{H}^m}+\int_0^t\|f(s)\|^2_{\mathcal{H}^{m-1}}ds.
\end{align}

\printbibliography
\end{document}